\documentclass[a4paper]{article}

\usepackage{amsmath}

\usepackage{amsthm}
\usepackage{graphicx}
\usepackage{epsfig}
\usepackage{enumerate}
\usepackage{amssymb}
\usepackage{latexsym}
\newtheorem{Theorem}{Theorem}
\newtheorem{Proposition}{Proposition}
\newtheorem{Lemma}{Lemma}
\newtheorem{Corollary}{Corollary}

\theoremstyle{definition}
\newtheorem{definition}{Definition}
\newtheorem{remark}{Remark}
\newtheorem{eg}{Example}
\newtheorem{qu}{Question}

\def\ZZ{{\mathbb Z}}

\def\NN{{\mathbb N}}

\def\S{{\boldsymbol{S}}}

\begin{document}

\title{Characterising knots by their branched coverings}

\author{Luisa Paoluzzi}
\date{November 2025}
\maketitle

\begin{center}
Workshop \emph{Branched Covers}
\\
IMUS -- Universidad de Sevilla
\end{center}

\begin{abstract}
\vskip 2mm
These are the notes for a mini-course on the problem of determining knots by 
means of their cyclic branched coverings given at the \emph{Workshop on
Branched Covers} held at IMUS in Seville from November 17 to 21, 2025.


\vskip 2mm


\end{abstract}

\section{Introduction}

Branched coverings are a convenient way to encode more complicated manifolds by
means of simpler ones (e.g. the sphere) together with a subset (the branching 
locus, e.g. a knot, link or graph) and some combinatorial data (ramification 
locus and monodromy). 

Recall that for $X$ and $Y$ topological spaces, if a map $p:X\longrightarrow Y$
is a branched covering then (i) $p$ is surjective and open, (ii) the subset
$R\subset X$ where $p$ is not a local homeomorphism is closed and nowhere
dense, and (iii) the map induced by $p$ from $X\setminus p^{-1}(p(R))$ to 
$Y\setminus p(R)$ is a genuine covering. With this notation $Y$ is the
\emph{base}, $X$ the \emph{total space}, $R$ the \emph{ramification locus}, and 
$B=p(R)$ the \emph{branching locus} of the covering. Depending on the context, 
one may require $B$ to satisfy extra ``regularity" conditions (for instance, a 
subcomplex of $Y$ of codimension $>1$ in the PL category, or a submanifold in 
the smooth category). 
 
Here, however, we are not interested in describing manifolds and will adopt a 
different point of view: we want to look at branched coverings as invariants 
for knots in the $3$-sphere. So the branching locus of the coverings we 
consider will be a knot\footnote{Other coverings will appear where the 
branching set is a link or a graph.} and the base of the covering the 
$3$-sphere. Moreover, our coverings will be finite so that their total spaces 
will be closed, connected, oriented $3$-manifolds. Other choices are of course 
possible, but we will focus on this setting. Before we proceed, though, it is 
necessary to make a few remarks in order to point out certain subtleties.
 
\subsection{Be aware of an abuse of language}

Branched coverings are maps. If $p:M\longrightarrow\S ^3$ is branched over a
knot $K$, then of course the covering is a complete invariant for $K$, as $K$ 
is part of the information encoded in the map: $K$ can be recovered as the 
image of the points where $p$ is not a covering in the usual, unbranched sense.
Such an invariant, however, does not have much interest as it is neither
simpler than $K$ itself, nor a totally different object that my shed some light
on, say, possible ways to classify knots. 

In the literature, by abuse of language, the expression \emph{branched covering 
of a knot} refers to the total space of the covering, that is $M$ for the
covering $p$ given above. In what follows we will also abuse language this way.

\subsection{Choose coverings appropriately}

In order to be advantageous, invariants should be easier to construct and
tell apart than the objects they should distinguish, or, at least, of a very
different nature, so that alternative tools and techniques can be exploited to
analyse them. For instance, in our situation, invariants are manifolds, i.e. of 
a topological or geometric nature, while the objects themselves, knots, 
exhibit a more combinatorial nature, due to their usual presentation in terms 
of diagrams, for instance.

But what coverings would be appropriate invariants? A possibility might be to
take the set of all of them (up to homeomorphism). However it is immediately 
clear that this option is definitely unbefitting: not only this set can be very
large, but the existence of infinitely many knots that are \emph{universal} 
shows that for infinitely many knots the set consists of all closed oriented 
$3$-manifolds. This follows from work of Hilden, Lozano, and Montesinos who 
proved that every two-bridge, non torus knot (actually link) has the property
\cite{HLM}. 

It is thus advisable to choose a much smaller family of coverings that are easy 
to construct. Better still if there is a way to compare different elements in
the family for distinct knots.

Keeping this in mind, (finite) regular coverings are among the simpliest to 
describe, for the base space is just the space of orbits of a properly
discontinous action of a group on a manifold and the branching locus is the
image of the set of points with non trivial stabilisers. Note also that 
branched coverings restrict to unbranched coverings of the complement of the 
knot, so that in order to construct them, one might start by considering 
(finite index) subgroups of the fundamental group of the knot or, in the case 
of regular ones when the subgroup is normal, (finite) quotients.

Even in this case, though, one might wonder what quotients of the fundamental
group of the knot are best suited. Since all knot groups have infinite
cyclic Abelianisation, for each integer $n\ge 2$ every knot $K$ admits a 
unique $n$-fold cyclic branched covering (see the next section for details
about how these are constructed). These are the coverings that we will consider
here.

\subsection{Banched coverings vs orbifolds}

Regular coverings have also another advantage. Since the base of the covering
is naturally endowed with an orbifold structure, most of the time the 
\emph{geometry} of the knot will be reflected in that of the total space. This
will be discussed in detail later. For now just note that the geometry of their 
exteriors (that is, hyperbolic, Seifert fibred or toroidal) provides a very 
rough first classification of knots. We end this discussion by recalling that
non-regular coverings in general cannot induce an orbifold structure on the 
base of the covering\footnote{A necessary condition for the base to have an
orbifold structure is that the preimage of the branched locus coincides with
the ramification locus.}. In the other direction not all orbifolds come from 
branched coverings, since bad orbifolds exist. Moreover, depending on the 
adopted definition of branched coverngs, the singular set of certain orbifolds 
cannot be the branched locus of a branched covering: this is for instance the 
case of reflection orbifolds, as the images of the reflection hyperplanes 
produce a topological ``boundary" in the base of the covering.    


\section{Construction of cyclic branched coverings of knots and basic 
properties}

Let now $K$ be a knot in the $3$-sphere and $n\ge 2$ be an integer. According
to the discussion in the previous section, the \emph{$n$-fold cyclic covering 
of $\S^3$ branched along $K$} is a closed, connected, orientable (even 
oriented) $3$-manifold $M(K,n)$ that can be described in basically two ways.

The first description exploits the fact that the branched covering restricts to 
a genuine cover from the complement of the ramification locus $\tilde K$ to the 
complement of the knot. $M(K,n)\setminus \tilde K$ is thus the total space of
the unique $n$-fold cyclic cover of $\S^3\setminus K$. Uniqueness comes from
the fact that $H_1(\S^3\setminus K)\cong\ZZ$ and there is a unique subgroup of
index $n$ in $\ZZ$. $M(K,n)$ is then simply the Dehn filling of such total
space (after a $\ZZ/n\ZZ$-equivariant open cusp is removed) with filling slope 
a lift of a meridian of $K$. As the covering is regular, the group $\ZZ/n\ZZ$ 
acts on $M(K,n)\setminus \tilde K$ with quotient $\S^3\setminus K$. The action 
extends to the solid torus of the filling as a standard rotation about its 
core, so that the core of such solid torus in $M(K,n)$ maps to $K$ in the 
quotient $\S^3$.

As a side remark, from this first description it is also clear that cyclic 
branched coverings of links can only be unique if $n=2$.

A different description can be found in Rolfen's book~\cite{R}. It has the 
advantage to offer a better visualisation of $M(K,n)$, but it cannot ensure 
uniqueness. One starts by cutting $\S^3$ along a Seifert surface for $K$. The 
resulting manifold has boundary consisting of two copies of the Seifert 
surface, say $S_-$ and $S_+$, that meet along $K$. It then suffices to take $n$ 
copies of such manifold and glue them together so that the $S_+$ in the $i$th 
copy is glued to the $S_-$ of the $i+1$st, with indices taken mod $n$. The
$\ZZ/n\ZZ$-action is now simply the cyclic permutation of the $n$ copies of the
manifold and the $\S^1$ which is the boundary of all the $S_-$ and $S_+$ is 
fixed by construction. This construction is particularly neat for fibred knots, 
as the covering inherits an open book decomposition with monodromy the 
$n$th-power of the monodromy for the knot. Unlike the fibred case, however, 
Seifert surfaces, even minimal ones, need not be unique (up to isotopy) so in 
principle one might end up with different manifolds.

The above discussion may be summarised in the following result which is a
characterisation of $n$-fold cyclic branched coverings for knots.

\begin{Proposition}
The $n$-fold cyclic branched covering of a knot $K$ in $\S^3$ is the
unique closed, connected, orientable $3$-manifold $M(K,n)$ admitting a periodic
diffeomorphism $\psi$ of order $n$ such that 
\begin{itemize}
\item ${\mathrm {Fix}}(\psi^k)\cong\S^1$ for all $k=1,\dots,n-1$, 
\item $M(K,n)/\langle\psi\rangle\cong \S^3$ and
\item $p({\mathrm {Fix}}(\psi))=K$, where $p:M(K,n)\longrightarrow \S^3$ is the 
quotient map.
\end{itemize}
\end{Proposition}

\begin{qu}
The above discussion applies if $\S^3$ is replaced with an integral homolgy
sphere. One may ask whether the results that will be stated in what follows 
make sense and have a chance to extend to this more general
setting.\footnote{The positive solution to Smith's conjecture being a central
tool in many arguments, it seems likely that most results will not hold for
knots in homology spheres.}
\end{qu}

\begin{eg}
Let $n\ge 2$ be an integer and $K$ be the trivial knot. Then $M(K,n)\cong 
\S^3$.
\end{eg}

More interestingly one has the following result which is a consequence of the
positive solution to Smith's conjecture\footnote{The conjecture states that if
the fixed-point set of a finite order diffeomorphism of the $3$-sphere is a
circle, then it is the trivial knot, that is the diffeomorphism is conjugate to
a standard rotation of the same order.} \cite{MB} and the proposition above.

\begin{Theorem}\label{ThmSC}
Let $n\ge 2$ be an integer and $K$ a knot. Assume that $M(K,n)\cong \S^3$. Then 
$K$ is the trivial knot.
\end{Theorem}

This observation prompts the following definitions. 

\begin{definition}
Let $K$ be a knot and $n\ge 2$ an integer. We say that $M(K,n)$
\emph{determines} $K$ if for every knot $K'$ such that $M(K,n)$ and $M(K',n)$
are homeomorphic, $K$ and $K'$ are equivalent. 

Otherwise, if there is a knot $K'$ non equivalent to $K$ such that $M(K,n)$ and 
$M(K',n)$ are homeomorphic we say that $K'$ is an \emph{$n$-twin} of $K$.
\end{definition}

Theorem~\ref{ThmSC} above says that for every integer $n\ge 2$, $M(K,n)$
determines the trivial knot, and the trivial knot has no $n$-twins. This can be
rephrased by saying that cyclic branched coverings are strong invariants for 
the trivial knot. Unfortunately not all knots behave like the trivial one with
respect to cyclic branched coverings and $n$-twins do exist. One might then 
hope to be able to determine a knot by considering several cyclic branched 
coverings at the time.

\begin{definition}
Let $F\subset\NN\setminus\{0,1\}$ be a set of integers and $K$ a knot. We say
that \emph{$K$ is determined by the family $(M(K,n))_{n\in F}$ of its
cyclic branched coverings} if whenever $M(K,n)$ and $M(K',n)$ are homeomorphic 
for all $n\in F$ then $K$ and $K'$ are equivalent. 
\end{definition}

It is thus natural to ask two different types of questions:

\begin{qu}
\begin{itemize}
\item Can a knot be determined by its cyclic branched covers? How many do we
need?
\item If $K$ has an $n$-twin $K'$, how can one recover $K'$ from $K$? 
\end{itemize}
\end{qu}

In the sequel we will address these two questions and present what is known
about them\footnote{A substantial part of these notes is based on
the survey \cite{P3} where more examples can also be found.}.


\section{Composite knots}

The following easy result is a first instance of the general principle that the 
geometry (in a large sense) of the knot is reflected in that of its cyclic 
branched coverings.

\begin{Lemma}\label{lCK}
Let $K$ be a knot and $n\ge 2$ an integer. $M(K,n)$ is a prime manifold if and 
only if $K$ is a prime knot. 
\end{Lemma}

\begin{proof}
We will reason by contraposition. Assume that $K$ is not prime and let $S$ be
an embedded sphere in $\S^3$ meeting $K$ in two points and separating it into
two non-trivial arcs realising $K$ as the connected sum of two non trivial
knots $K_1$ and $K_2$, that is $K=K_1\sharp K_2$. Such sphere lifts in $M(K,n)$
to a sphere separating the manifold in two submanifolds which are obtained
from $M(K_1,n)$ and $M(K_2,n)$ by removing a ball, so that
$M(K,n)=M(K_1,n)\sharp M(K_2,n)$. This shows that the branched covering is a
non-trivial connected sum according to Theorem~\ref{ThmSC}. 

For the other implication, by the equivariant sphere theorem, on can assume
to have a decomposition of $M(K,n)$ that is equivariant. Since every $2$-sphere
in $\S^3$ bounds a ball and $\S^3$ does not contain any essential real
projective plane, every sphere of the decomposition of $M(K,n)$ is separating 
and its image in the quotient must intersect $K$. This shows that $K$ is not
prime.
\end{proof}

It follows from the argument in the proof that the $n$-fold branched covering 
of a composite knot is the connected sum of the $n$-fold branched coverings of 
its components. The following essential observation is due to Viro \cite{V,V1}.

\begin{remark}
There are two possible ways to define the connected sum of two knots, according
to the possible orientations of the two knots and disregarding the orientation 
of the resulting ones. The latter are indeed disctinct if their components are
non invertible knots. On the other hand, the $n$-fold cyclic branched coverings 
of the resulting composite knots do not depend on the orientations of the 
original knots. As a consequence there are non equivalent composite knots that 
are $n$-twins for every $n\ge 2$. Moreover, the number of $n$-twins with this
property can be arbitrarily large as it grows with the number of prime
components of the knot.
\end{remark}

\begin{figure}[h]
\begin{center}
  \includegraphics[width=10cm]{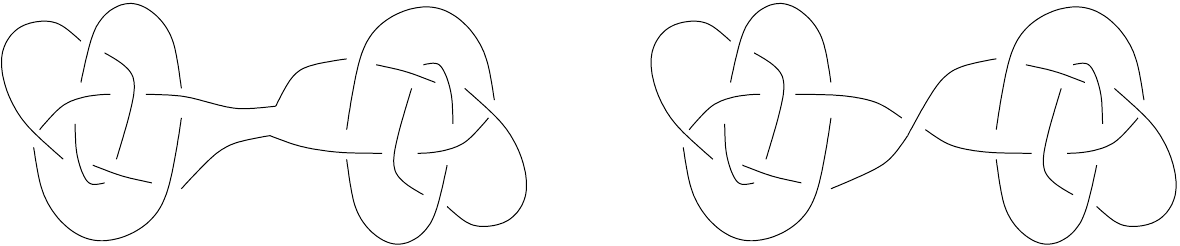}
  \caption{Two composite knots that are $n$-twins for all $n\ge 2$.}
\end{center}
\end{figure}

Roughly speaking the above remark says that one can only hope to determine the
components of a composite knot using cyclic branched coverings, but not the 
knot itself. As a consequence, from now on all knots will be assumed to be 
prime. Note that Lemma~\ref{lCK} assures that a prime knot cannot have a 
composite twin.

A result of Kojima \cite{Ko} assures that cyclic branched coverings are 
reasonably good invariants for prime knots:

\begin{Theorem}[Kojima]
Let $K$ and $K'$ be prime knots. There exists an integer $N$, depending on $K$
and $K'$, such that if $M(K,n)\cong M(K',n)$ for an $n\ge N$ then $K$ and $K'$
are equivalent. 
\end{Theorem}

In other words, any two non-equivalent prime knots $K$ and $K'$ cannot be
$n$-twins for $n$ large enough.  


\section{How to build twins}

It is natural to ask whether the $N$ in Kojima's result can be chosen large
enough to be made independent of $K$ and $K'$. Nakanishi and Sakuma
\cite{N,Sa1} answered this question in the negative by building the first 
examples of $n$-twins for arbitrary $n$. Their construction is as follows.

Take a two-component link $L$ with trivial components $L_1$ and $L_2$. We can
consider the $\ZZ/n\ZZ\times \ZZ/n\ZZ$ branched covering of $L$ to obtain a
manifold $M$. Now, since $\ZZ/n\ZZ\times \ZZ/n\ZZ$ is Abelian, the covering can 
be obtained as a composition of two $n$-fold cyclic coverings in two different
ways. One starts with taking the $n$-fold cyclic coverings of $L_1$. Since this
is a trivial knot, $M(L_1,n)$ is the $3$-sphere. The preimage $K_2$ of $L_2$ in
$M(L_1,n)$ is a link with ${\mathrm{GCD}}({\mathrm{lk}}(L_1,L_2),n)$ components 
(as ${\mathrm{lk}}(L_1,L_2)$ is precisely the image of $L_2$ in the 
Abelianisation of $\pi_1(\S^3\setminus L_1)$). So, if
${\mathrm{GCD}}({\mathrm{lk}}(L_1,L_2),n)=1$ (in particular if
$|{\mathrm{lk}}(L_1,L_2)|=1$), then $K_2$ is a knot and, by construction, one
has $M(K_2,n)\cong M$. Exchanging the roles of $L_1$ and $L_2$ one obtains two
knots, $K_1$ and $K_2$, with the same $n$-fold cyclic branched covering. If $L$
has no symmetries, one can expect $K_1$ and $K_2$ to be non-equivalent and so 
$n$-twins.  

\begin{figure}[h]
\begin{center}
  \includegraphics[width=10cm]{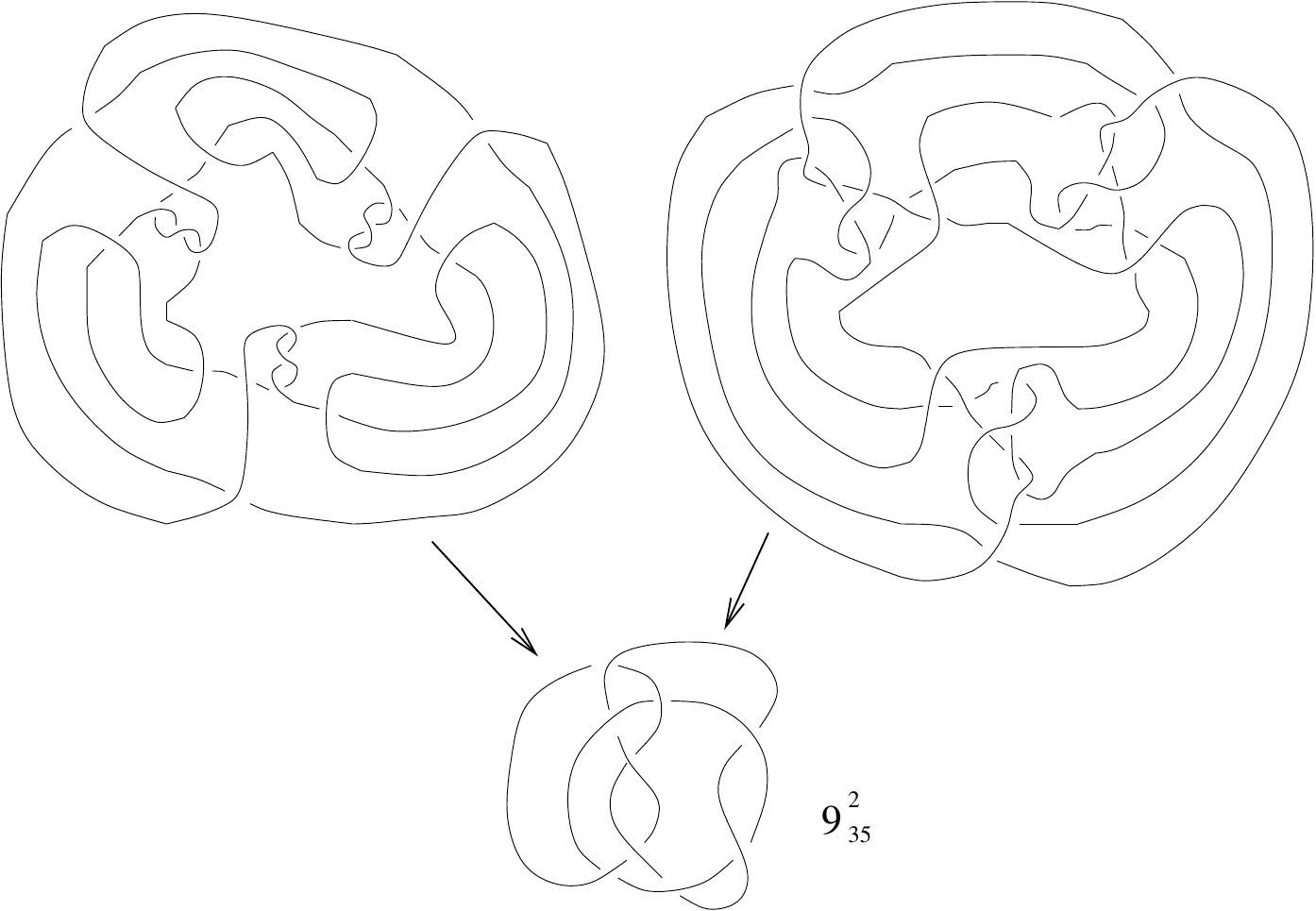}
  \caption{An example of Nakanishi and Sakuma's construction for $n=3$}
\end{center}
\end{figure}

Note that by taking $\ZZ/n\ZZ\times \ZZ/m\ZZ$ instead of $\ZZ/n\ZZ\times
\ZZ/n\ZZ$ one can construct manifolds that are cyclic branched coverings of 
knots for different degrees of the covering.

\subsection{The case $n\ge 3$}

This simple method is particularly remarkable as Zimmermann \cite{Z2} showed 
that this is the only way for a hyperbolic knot to admit an 
$n$-twin\footnote{In \cite{Z2} $n$ is assumed not to be a power of $2$, but a
variation of the same argument applies also to the case $n=2^k$ (see 
\cite{P2}).} for $n\ge 3$. Moreover, if such an $n$-twin exists, because of the 
positive solution to Smith's conjecture, the twin is unique in this case.

\begin{Theorem}[Zimmermann]
Let $K$ be a hyperbolic knot and $n\ge 3$ an integer. $K$ admits an $n$-twin
$K'$ if and only if $K$ admits a \emph{period} of order $n$ with \emph{trivial
quotient} such that the two components of the \emph{quotient link} are not
exchangeable. 
\end{Theorem}

\begin{definition}
Let $K$ be a knot and $n\ge 2$. A \emph{period of order $n$} of $K$ is an
orientation-preserving, order $n$, diffeomorphism $\phi$ of the pair $(\S^3,K)$ 
such that ${\mathrm{Fix}}(\phi)$ is a circle disjoint from $K$. It has
\emph{trivial quotient} if the image of $K$ in the quotient $\S^3=S^3/\langle 
\phi \rangle$ is the trivial knot. The \emph{quotient link} is the image of
$K\cup {\mathrm{Fix}}(\phi)$ in the quotient.
\end{definition}

The ingredients in this result are basically two: cyclic branched coverings
reflect the geometry of the knot and the isometry group of a hyperbolic 
manifold is finite. The first point is made precise in the following remark.

\begin{remark}\label{rkGeo}
Let $K$ be a prime knot and $n\ge 2$. If $K$ is a torus knot then $M(K,n)$ is
Seifert fibred. If $K$ is a satellite knot, then $M(K,n)$ is toroidal (in fact
admits a non trivial JSJ decomposition). If $K$ is hyperbolic then, by
Thurston's orbifold theorem \cite{BLP}, $M(K,n)$ is hyperbolic provided that 
either $n\ge 4$ or $n=3$ and $K$ is not the figure-eight knot. If $n=2$, 
$M(K,n)$ can be hyperbolic, Seifert fibred, or admit a non trivial JSJ 
decomposition.  
\end{remark}

This remark says in particular that the twin $K'$ in Zimmermann's result is 
also a hyperbolic knot. 

Assuming now that a hyperbolic knot $K$ has an $n$-twin $K'$ for some $n\ge 3$, 
one starts by verifying that $K$ cannot be the figure-eight knot, so that
$M(K,n)$ is hyperbolic. In particular the automorphisms $\psi$ and $\psi'$ of 
the branched coverings for $K$ and $K'$ can be assumed to be hyperbolic
isometries and so elements of the same finite group. Elementary finite group 
theory (namely, Sylow $p$-subgroups) and geometric considerations lead to the 
conclusion that, up to choosing appropriate conjugates, $\psi$ and $\psi'$ 
commute.

For prime satellite knots, one cannot hope to be able to choose $\psi$ and
$\psi'$ so that they commute for all $n$-twins $K$ and $K'$. Explicit examples
of $n$-twins that do not behave like in Nakanishi and Sakuma's construction are 
known for $n$ an odd prime number \cite{BP}. Nonetheless one has the following
result \cite{BP}.

\begin{Theorem}[Boileau-P.]
Let $K$ be a prime knot and $n\ge 3$ a prime integer. $K$ admits at most one
$n$-twin and for at most two values of $n$. 
\end{Theorem}

\begin{Corollary}
Let $K$ be a prime knot. The family $(M(K,3),M(K,5),M(K,7))$ determines $K$.
\end{Corollary}

\subsection{The case when $n=2$}

As one may guess from Remark~\ref{rkGeo} above, there is a plethora of possible 
ways to obtain $2$-twins, even of hyperbolic knots. We will discuss different 
cases according to the geometry of $M(K,2)$.

\subsubsection{$M(K,2)$ is hyperbolic.}

In this case, just like in the previous section, the existence of a $2$-twin
for a hyperbolic knot entails the presence of symmetris of the knot. In this
case, the symmetries need not be periods: strong involutions appear and the
two-component link of the construction at the beginning of the section can be
replaced with a theta-curve (examples of this type were first studied by
Zimmermann \cite{Z1}). Also, taking a single quotient by a symmetry need not be
sufficient to ``see" all the $2$-twins of the knot. An accurate study of the
Sylow $2$-subgroup gives the following result \cite{Re2} (see \cite{MR} for a
more combinatorial approach by considering successive quotients).

\begin{Theorem}[Reni, Mecchia-Reni]
Let $K$ be a hyperbolic knot and assume that $M(K,2)$ is hyperbolic. $K$ has at
most eight $2$-twins.
\end{Theorem}
 
Using Kawauchi's strongly almost identical (AID) imitation theory one can show 
that nine $2$-twins must exist \cite{Ka}, but explicit constructions are only 
known for sets of two, three, and four twins \cite{RZ1}. 

\begin{figure}[h]
\begin{center}
  \includegraphics[width=6cm]{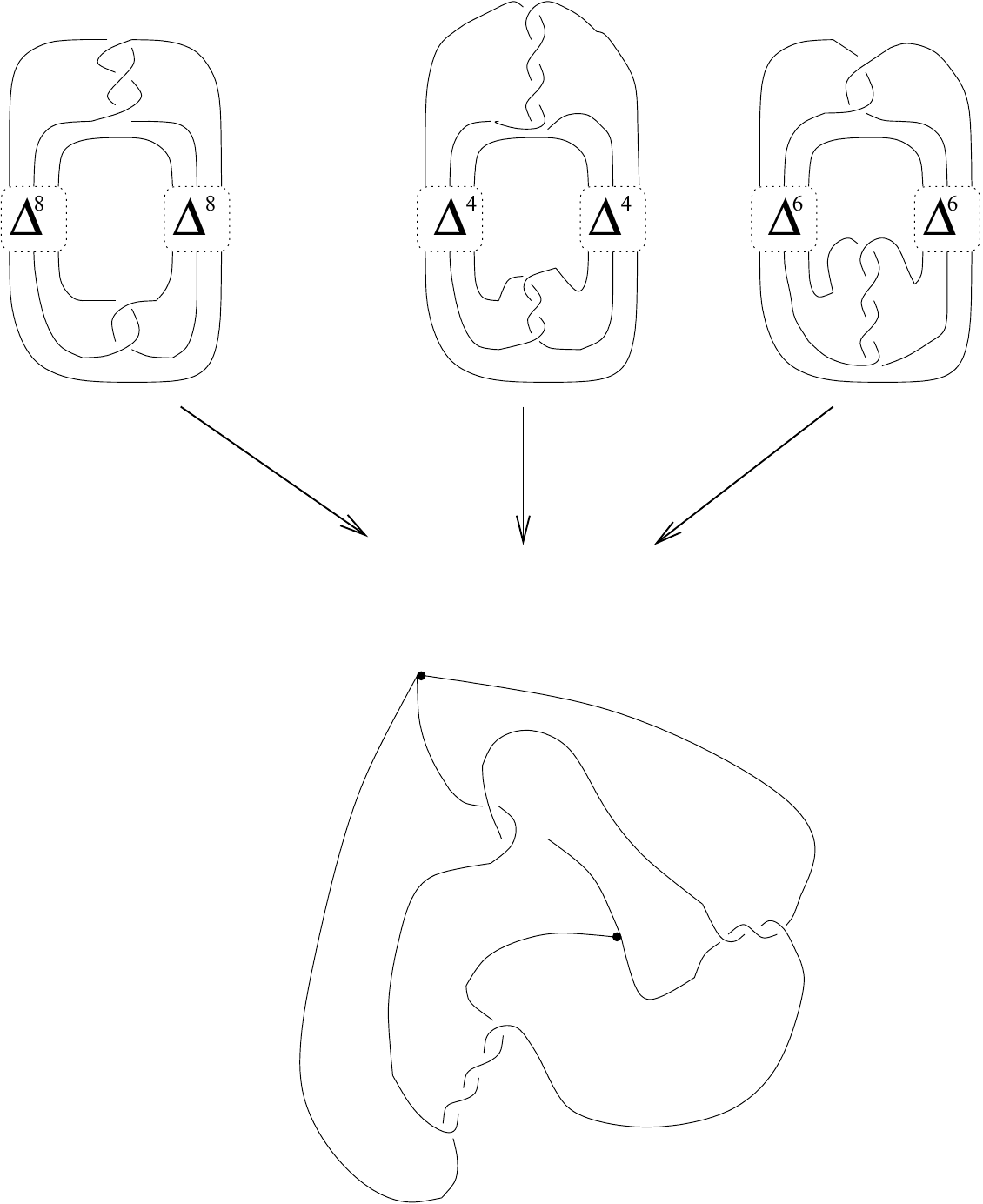}
  \caption{Three $2$-twins related by strong inversions and having quotient a
theta-curve.}
\end{center}
\end{figure}

\subsubsection{$M(K,2)$ is Seifert fibred.}

\begin{figure}[h]
\begin{center}
  \includegraphics[width=10cm]{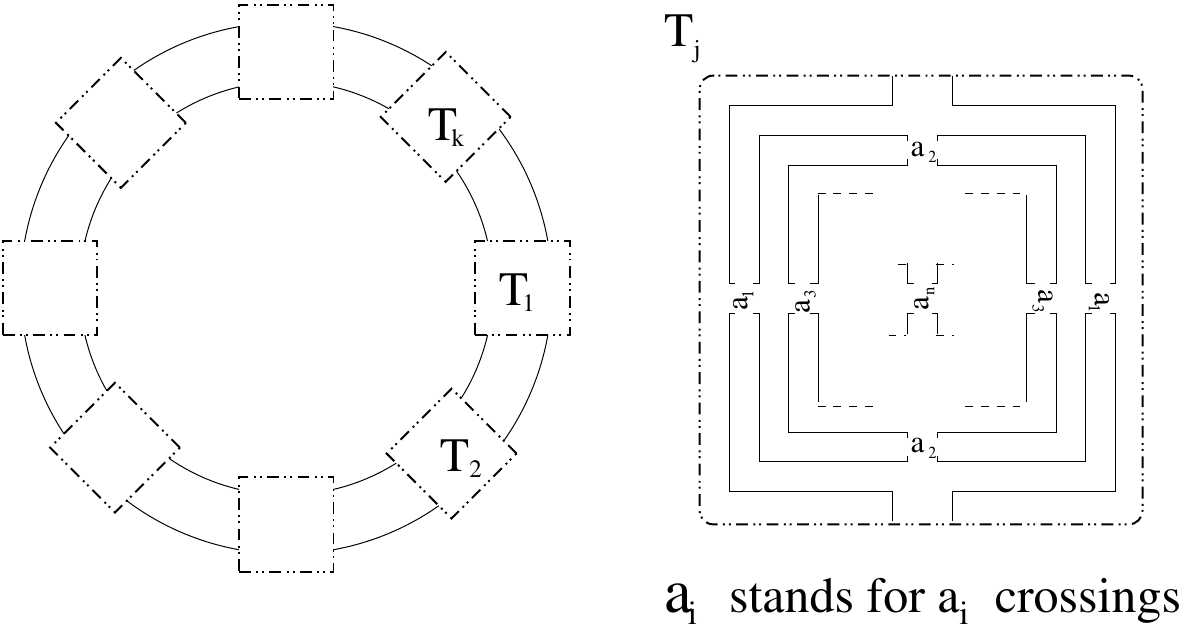}
  \caption{A schematic picture of a Montesinos link obtained by stacking 
together rational tangles, on the left, and one of a rational tangle, on the
right.}
\end{center}
\end{figure}

\begin{figure}[h]
\begin{center}
  \includegraphics[width=6cm]{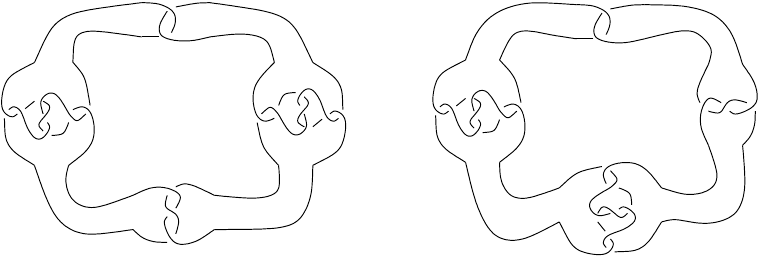}
  \caption{Two Montesinos knots that are $2$-twins.}
\end{center}
\end{figure}

In this case $K$ is a Montesinos knot \cite{Mo}. The number of $2$-twins of a 
Montesinos knot can be arbitrarily large as two Montesinos knots having the 
same rational tangles have the same $2$-fold branched covering regardless of 
the order in which their rational tangles are arranged. Indeed, each rational 
tangle lifts in $M(K,2)$ to the neighbourhood of an exceptional fibre and 
$M(K,2)$ is independent of their order. 

If the Montesinos knot has precisely three rational tangles it does not have
any hyperbolic $2$-twins but may have a $2$-twin which is a torus knot.

\subsubsection{$M(K,2)$ has a non-trivial JSJ decomposition.}

Since the complement of $K$ is atoroidal, the covering tranformation must have
a fixed-point set that intersect all essential tori of the JSJ decomposition of 
$M(K,2)$ in four points and must act as an elliptic involution on 
each of them. Such tori give rise to essential Conway spheres for the knot. In 
this case there is a well-known method to produce $2$-twins: Conway mutation. 
Conway mutation along a Conway sphere that separates two non-trivial 
non-symmetric tangles results in a $2$-twin, as for the Conway and 
Kinoshita-Terasaka knots.

\begin{figure}[h]
\begin{center}
  \includegraphics[width=8cm]{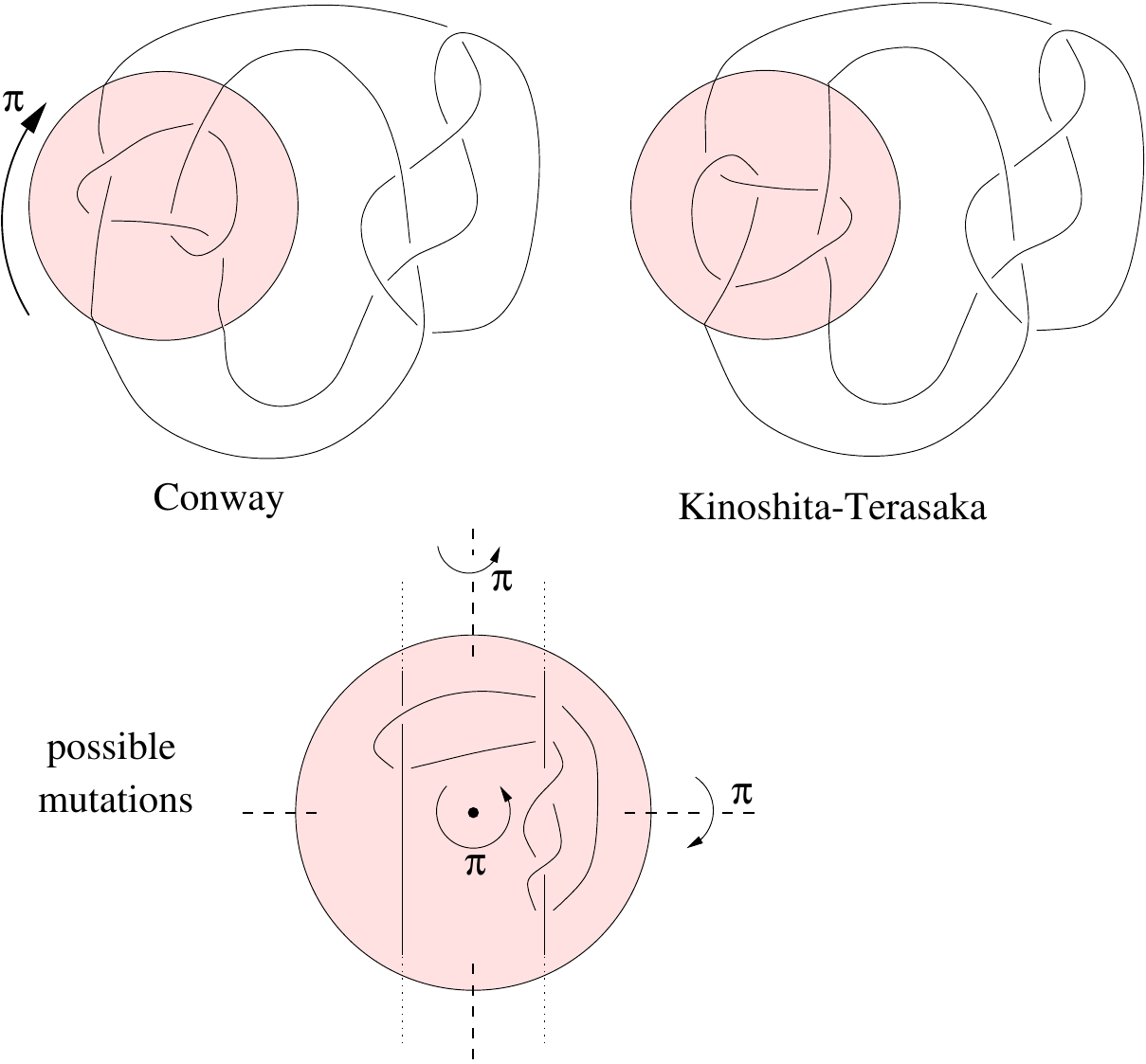}
  \caption{The Conway and Kinoshita-Terasaka knots and possible mutations along
a Conway sphere.}
\end{center}
\end{figure}

Note that two Montesinos knots that are $2$-twins are also obtained by a
sequence of Conway mutations.

\begin{figure}[h]
\begin{center}
  \includegraphics[width=6cm]{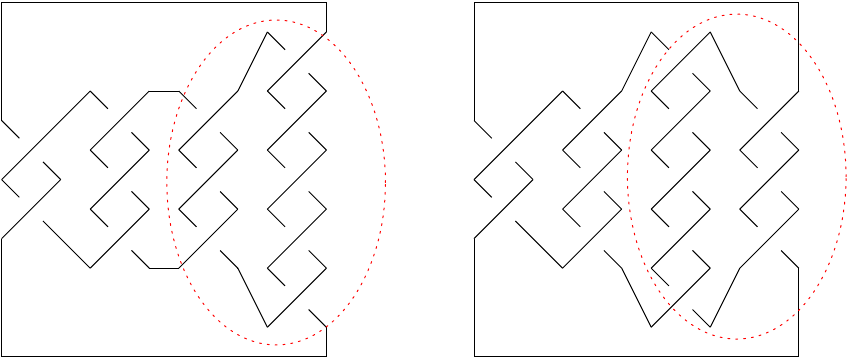}
  \caption{Two Montesinos knots that are $2$-twins in which the Conway mutation
is shown.}
\end{center}
\end{figure}

Conway mutation is not the end of the story, though. In this case $2$-twins can
arise from a combination of the possibilities seen so far (mutation,
permutation of rational tangles, symmetries that might be localised on some 
geometric part) and can even be satellite knots \cite{P1}. The first examples 
of twins of this type can be found in a paper by Montesinos and Whitten
\cite{MW}.

\begin{figure}[h]
\begin{center}
  \includegraphics[width=10cm]{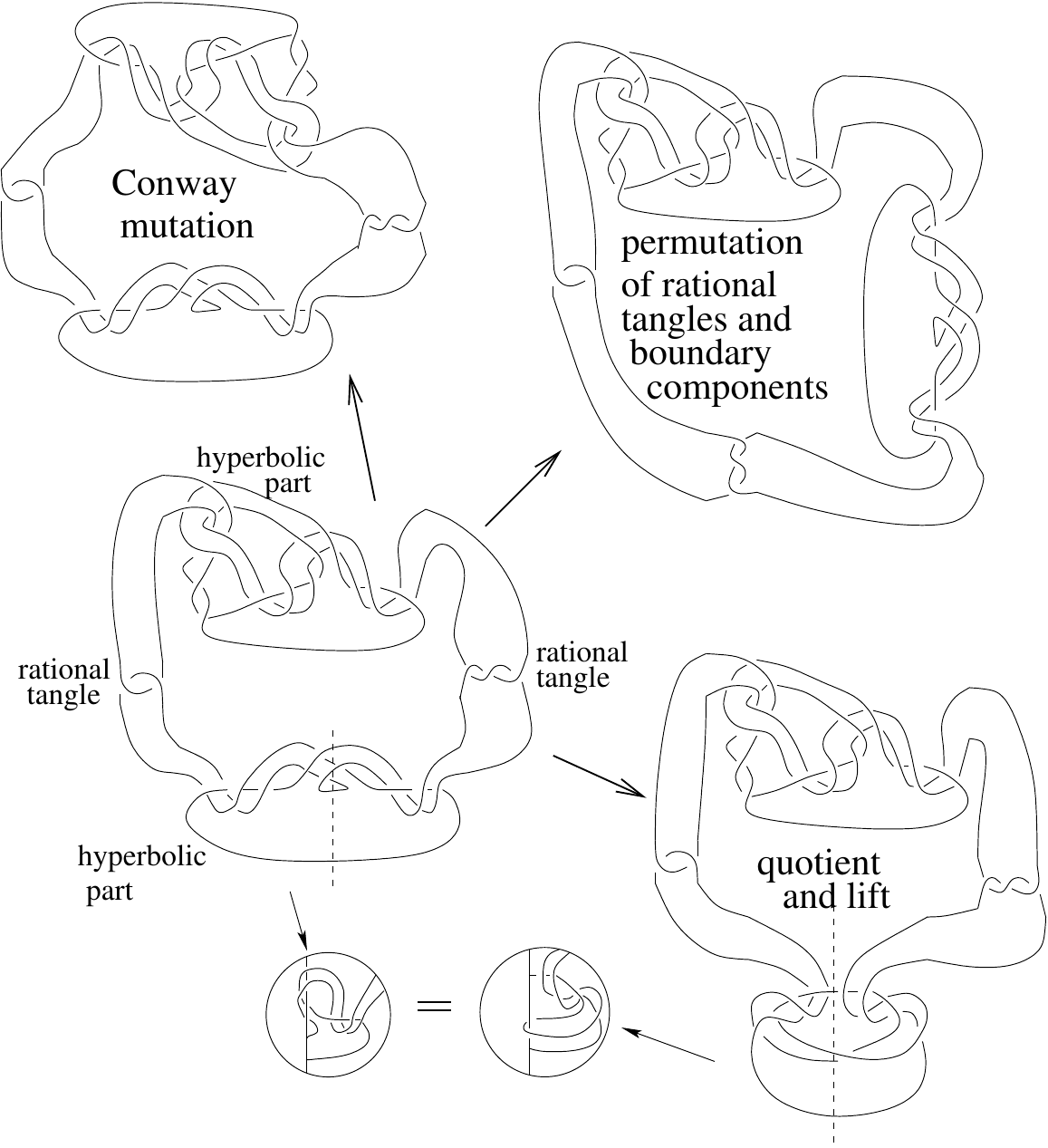}
  \caption{Hyperbolic $2$-twins obtained by mutations and local
Nakanishi-Sakuma's method.}
\end{center}
\end{figure}

\subsection{The case of alternating knots.}

In spite of the wide variety of constructions mentioned in this section,
$2$-twins of alternating knots only come from mutation if they are themselves
alternating, as shown by Greene \cite{G1}.

\begin{Theorem}[Greene]
If $K$ and $K'$ are prime alternating knots which are $2$-twins, then $K$ and
$K'$ are obtained by a sequence of Conway mutations.
\end{Theorem}

Greene conjectured something even more surprising, that is that a $2$-twin of
a prime alternating knot must be alternating.

\subsection{Presence of symmetries as an indicator of the existence of twins}

We have seen that the existence of a twin for a knot may imply that the knot is
symmetric. In some instances the presence of symmetries is enough to ensure
that the knot has a twin. A special class to consider is that of 
\emph{tunnel-number-$1$ knots}. These are the knots whose exteriors admit a
Heegaard splitting of genus $2$. Because of this, as surfaces of genus $2$ are
hyperelliptic, they are all strongly invertible. We have the following
\cite{JP}.

\begin{Theorem}[Jang-P]
Let $K$ be a tunnel-number-$1$ knot. If $K$ satisfies one of the following
conditions, it has a $2$-twin:
\begin{itemize}
\item the bridge index of $K$ is $\ge 5$.
\item the bridge index of $K$ is $4$ and $K$ is a $(1,1)$-knot\footnote{A
\emph{$(1,1)$-knot} is a knot that can be put in $1$-bridge position with
respect to a genus-$1$ Heegaard splitting of the $3$-sphere. Torus knots and
$2$-bridge knots are examples of $(1,1)$-knots.} (or $M(K,2)$ has Heegaard 
genus equal to $2$).
\end{itemize}
\end{Theorem}


\section{Rigidity results}

Although $2$-fold branched coverings seem too flexible to be reasonable 
invariants for knots, sometimes they are enough to determine them.

\begin{Theorem}[Hodgson-Rubinstein]
Let $K$ be a $2$-bridge knot. Then $K$ has no $2$-twins, that is $M(K,2)$
determines $K$.
\end{Theorem}

This follows from Hodgson and Rubinstein's classification of involutions on
lens spaces in \cite{HR}.

In the case of torus knots, a careful combinatorial analysis of the Seifert 
invariants of their cyclic branched coverings, which are Brieskorn manifolds 
and were classified by Neumann in his Ph.D. thesis, gives the following result.

\begin{Proposition}
A torus knot cannot be the twin of another torus knot.
\end{Proposition}

\begin{Corollary}
Let $K$ be a torus knot and $n\ge 3$. Then $K$ is determined by $M(K,n)$.
\end{Corollary}

This means that torus knots can only have $n$-twins (which are Montesinos
knots) for $n=2$ and this can only happen if they are not of type $T(2,2a+1)$.
Note that in a couple of peculiar cases, the Montesinos knot is not hyperbolic  
and coincides with the torus knot which thus has no $2$-twin.

We have seen that an arbitrary prime knot is determined by a family of three
cyclic branched coverings, provided their degrees are odd prime numbers. If 
this requirement can be improved for torus knots as mentioned above, for 
hyperbolic knots one can relax the requirement on the degrees or reduce the 
size of the family by imposing conditions on the degrees. Indeed we have the 
following results.

\begin{Theorem}[P., Zimmermann]
Let $F\subset\NN\setminus\{0,1\}$ and let $K$ be a hyperbolic knot. The family
$(M(K,n))_{n\in F}$ determines $K$ if one of the following conditions is
satisfied.
\begin{enumerate}
\item The cardinality of $F$ is $3$ \cite{P2}.
\item $F=\{n,m\}$ and 
\begin{enumerate}
\item $n,m\ge 3$ and are not coprime \cite{Z2};
\item $n>m=2$ and 
\begin{itemize}
\item $K$ is a Montesinos knot \cite{P} or 
\item $M(K,2)$ is hyperbolic and $n$ is even \cite{Z2} or
\item $K$ has a unique Seifert surface, for instance $K$ is fibred \cite{Pa}.
\end{itemize}
\end{enumerate}
\item $F=\{n\}$, $n\ge 3$, and $K$ is alternating \cite{P4}.
\end{enumerate}
\end{Theorem}

This result is best possible as for every pair of integers $n>m$ one can 
construct non-equivalent hyperbolic knots $K$ and $K'$ that are $n$-twins and
$m$-twins provided that one of the following conditions hold:

\begin{enumerate}
\item $m\ge 3$ and $m$ and $n$ are coprime \cite{Z1};
\item $m=2$ \cite{P,Pa}.
\end{enumerate}

Keeping in mind the construction to build pairs of $n$-twins, one realises that
to obtain pairs of knots that are $n$-twins and $m$-twins for two coprime
integers $n$ and $m$, one can start with a three-component (hyperbolic) link 
such that all of its components are trivial, every two-component sublink is a 
Hopf link, and such that its three components can be cyclically permuted but no 
two can be exchanged. The knots are obtained as lifts of the third component of
the link in the $\ZZ/n\ZZ \times \ZZ/m\ZZ$ and $\ZZ/m\ZZ\times \ZZ/n\ZZ$
branched coverings of the other two.


\section{A different point of view and an application}

We have seen that a manifold can be the $n$-fold cyclic branched covering of
arbitrarily many $n$-twin knots. One might now ask whether there is a universal
bound on the number of possible degrees $n$ such that a manifold $M$ not
homeomorphic to $\S^3$ is the $n$-fold branched covering of some knot in the 
$3$-sphere. The answer to this question is negative in general, however it is
possible to provide bounds by imposing conditions on the geometry of the
manifolds or on the degrees \cite{BFM}.

\begin{Theorem}[Boileau-Franchi-Mecchia-P.-Zimmermann]
\begin{itemize}
\item A hyperbolic manifold is the branched covering of some knot for at most
nine different degrees.
\item A manifold not homeomorphic to $\S^3$ is the branched covering of some 
knot for at most six odd prime degrees.
\end{itemize}
\end{Theorem}

\begin{Corollary}
A hyperbolic manifold is the branched covering of at most fifteen 
non-equivalent knots.
\end{Corollary}

The proof of this result relies on the classification of finite simple groups.

\subsection{Two characterisations}

The above result can be restated to provide a characterisation of the
$3$-sphere.

\begin{Corollary}
The following are equivalent:
\begin{itemize}
\item $M$ is the $3$-sphere;
\item $M$ is the cyclic branched covering of some knot in $\S^3$ for at least 
seven odd prime degrees;
\item $M$ admits a finite group of orientation-preserving diffeomorphisms
which contains at least sixteen conjugacy classes of elements with connected
fixed-point set and space of orbits homeomorphic to $\S^3$.
\end{itemize}
\end{Corollary}

In a similar spirit one can prove that an integral homology sphere is $\S^3$ if
and only if it is the cyclic cover branched over some knot in $\S^3$ for at  
least four odd prime degrees \cite{BPZ}. In this situation the result is best 
possible as the Brieskorn sphere with three exceptional fibres of orders 
$p_1>p_2>p_3\ge 3$, three prime integers, is the $p_i$-fold cyclic branched 
covering of the $T(p_j,p_k)$-torus knot for $\{i,j,k\}=\{1,2,3\}$. The question 
to know whether the bound six in the previous general case is optimal remains
open.

As we have seen, sometimes the existence of an $n$-twin for a knot induces an
$n$-period with trivial quotient. As a by-product of this we have the following
result \cite{BP} whose proof is elementary when the periods commute.

\begin{Theorem}[Boileau-P.]
A knot is trivial if and only if it admits three periods with trivial quotients
and pairwise distinct orders $>2$.
\end{Theorem}

\subsection*{Acknowledgement}

The author would like to thank M. Kegel and M. Silvero for the invitation to 
speak and careful organisation, as well as all participants to the workshop for 
contributing to a stimulating research environment and friendly atmosphere.

\textsc{Aix-Marseille Univ, CNRS, I2M UMR 7373, Marseille, France}

\texttt{luisa.paoluzzi@univ-amu.fr}

\end{document}